\documentclass[11pt,letterpaper,reqno]{amsart}
\usepackage{fullpage}
\usepackage[top=2cm, bottom=4.5cm, left=2.5cm, right=2.5cm]{geometry}
\usepackage{amsmath,amsthm,amsfonts,amssymb,amscd,systeme}
\usepackage{geometry,tikz,tikz-cd}
\usepackage{empheq}
\usepackage{lipsum}
\usepackage{lastpage}
\usepackage{thmtools}
\usepackage{enumerate}
\usepackage[shortlabels]{enumitem}
\usepackage{fancyhdr}
\usepackage{mathrsfs}
\usepackage{xcolor}
\usepackage{graphicx}
\usepackage{mathtools}
\usepackage{listings}
\usepackage{bbm}
\usepackage{hyperref}
\usepackage[makeroom]{cancel}
\usepackage[utf8]{inputenc}

\usepackage[all]{xy}
\xyoption{matrix}
\xyoption{arrow}

\usetikzlibrary{arrows,decorations.pathmorphing,backgrounds,positioning,fit,petri}

\tikzset{help lines/.style={step=#1cm,very thin, color=gray},
help lines/.default=.5} 
\tikzset{thick grid/.style={step=#1cm,thick, color=gray},
thick grid/.default=1} 

\mathtoolsset{showonlyrefs=true}
\numberwithin{figure}{section}
\numberwithin{table}{section}

\theoremstyle{definition}

\theoremstyle{plain}
\newcommand{\thistheoremname}{}
\newtheorem*{genericthm*}{\thistheoremname}
\newenvironment{namedthm*}[1]
  {\renewcommand{\thistheoremname}{#1}%
   \begin{genericthm*}}
  {\end{genericthm*}}

\hypersetup{%
  colorlinks=true,
  linkcolor=blue,
  citecolor=blue,
  linkbordercolor={0 0 1}
}

\allowdisplaybreaks

 \newtheorem{theorem}{Theorem}[section]
 \newtheorem{corollary}{Corollary}[section]
 \newtheorem{lemma}[theorem]{Lemma}
 \newtheorem{proposition}[theorem]{Proposition}
 \theoremstyle{definition}
 
 \theoremstyle{definition}
 
 \theoremstyle{remark}
 \newtheorem{remark}{Remark}[section]
\newtheorem{conjecture}[theorem]{\bf Conjecture}

 \numberwithin{equation}{section}

 \theoremstyle{remark}
 \newtheorem*{Notations*}{Notations}
 \newtheorem*{Proof*}{Proof}

\usepackage{latexsym}
\usepackage{amssymb}
\usepackage{nicematrix}

\newcommand{\ben}{\begin{equation}}
\newcommand{\een}{\end{equation}}

\DeclareMathOperator{\Sl}{SL}

\DeclareMathOperator{\Span}{Span}
\DeclareMathOperator{\Hom}{Hom}

\DeclareMathOperator{\re}{Re}

\DeclareMathOperator{\sign}{sign}

\setlist[enumerate]{leftmargin=*,widest=0}
\setlist[itemize]{leftmargin=*,widest=0}

\makeatletter
\def\subsection{\@startsection{subsection}{2}%
  \z@{.5\linespacing\@plus.7\linespacing}{.3\linespacing}%
  {\normalfont\bfseries}}

\def\subsubsection{\@startsection{subsubsection}{3}%
  \z@{.5\linespacing\@plus.7\linespacing}{.3\linespacing}%
  {\normalfont\bfseries}}
\makeatother

\title{Linear independence of periods for the symmetric square $L$-functions}

\subjclass[2020]{11F11, 11F67}
\keywords{}
\begin{document}
\author{Tianyu Ni}
\address{School of Mathematical and Statistical Sciences\\
Clemson University\\
Clemson, SC 29634-0975\\
USA}
\email{tianyuni1994math@gmail.com}
\author{Hui Xue}
\address{School of Mathematical and Statistical Sciences\\
Clemson University\\
Clemson, SC 29634-0975\\
USA}
\email{huixue@clemson.edu}
\begin{abstract} For $S_k$, the space of cusp forms of weight $k$ for the full modular group, we first introduce periods on $S_k$ associated to symmetric square $L$-functions. We then prove that for a fixed natural number $n$, if $k$ is sufficiently large relative to $n$, then any $n$ such periods are linearly independent. With some extra assumption, we also prove that for $k\geq e^{12}$, we can always pick up to $\frac{\log k}{4}$ arbitrary linearly independent periods.

\medskip

\noindent\textsc{R\'esum\'e.}
Pour $S_k$, l’espace des formes parabolique de poids $k$ pour le groupe modulaire complet, nous d’abord
introduire des p\'eriodes sur $S_k$ associ\'ees \`a des $L$-fonctions  carr\'ees sym\'etriques. Nous prouvons alors que pour un
nombre naturel fix\'e $n$, si $k$ est suffisamment grand par rapport \`a $n$, alors $n$ de telles p\'eriodes sont lin\'eaires
ind\'ependant. Avec quelques hypoth\`eses suppl\'ementaires, nous prouvons \'egalement que pour $k\ge e^{12}$, nous pouvons toujours choisir
jusqu'\`a  $\frac{\log k}4
$ p\'eriodes arbitraires lin\'eairement indépendantes.

\end{abstract}
\maketitle
\section{Introduction}
In a paper from 1979, Deligne \cite{Delignemotive} proposed a conjecture on the algebraicity of values of motivic $L$-functions at critical points in terms of motivic periods. A special family of examples are the symmetric power $L$-functions associated to modular forms. For each even integer $k\geq4$, let $M_k$ be the space of modular forms of weight $k$ for the full modular group $\Sl_2(\mathbb{Z})$, and let $S_k$ be its subspace of cusp forms. Let $f(z)=\sum_{n=1}^{\infty}a(n)e^{2\pi inz}\in S_k$ be a normalized Hecke eigenform such that $a(1)=1$. For each prime $p$, we denote by $\alpha_p$ and $\beta_p$ the roots of the polynomial
$$X^2-a(p)X+p^{k-1}.$$ 
Let $m\geq1$ be an integer. The $m^{{\rm th}}$ symmetric power $L$-function of $f$ for $\re(s)$ sufficiently large is defined as
\begin{align}\label{eq:symmetricpowerL}
    L_{m,f}(s)=\prod_{p}\prod_{0\leq i\leq m}(1-\alpha_p^{m-i}\beta_p^i\cdot p^{-s})^{-1},
\end{align}
and these functions are conjectured \cite{SerreconjectureaboutsymL} to have an analytic continuation to the whole complex plane and to satisfy the following functional equation for the completed $L$-function:
\begin{align}    L^{\ast}_{m,f}(s)&:=\gamma_m(s)\cdot L_{m,f}(s)=\pm L^{\ast}_{m,f}((k-1)m+1-s),\label{eq:functionalequation}\end{align}
    where $\gamma_m$ is given by
    \begin{align}
    \gamma_m(s):&=\begin{cases}(2\pi)^{-rs}\prod_{j=0}^{r-1}\Gamma(s-j(k-1))&m=2r-1,\\\pi^{-\frac{s}{2}}\Gamma\left(\frac{s}{2}-\lfloor\frac{
    r(k-1)}{2}\rfloor\right)\gamma_{2r-1}(s)&m=2r.\label{eq:gammafactor}
    \end{cases}
    \end{align}
See also \cite[p.177-178]{Zagier1977M} for details. 

In particular, if $m=1$, then \eqref{eq:functionalequation} is known as the functional equation of the standard $L$-function of the Hecke eigenform $f$ in $S_k$. The case $m=2$ was proved by Shimura \cite{functionalequationforL(Sym^2)}, where the functional equation \eqref{eq:functionalequation} is
\begin{align}
L^{\ast}_{2,f}(s)=2^{-s}\pi^{-3s/2}\Gamma(s)\Gamma\left(\frac{s-k+2}{2}\right)L_{2,f}(s)=L^{\ast}_{2,f}(2k-1-s).  \label{eq:functionalequationforsquareL}
\end{align}
Recently, Newton and Thorne \cite{Newton20211,Newton20212} showed that the analytic continuation holds true for all eigenforms $f$.
Deligne's conjecture for $L^{\ast}_{m,f}(s)$ at critical points  have been considered by various authors. For instance, the case $m=1$ was proved by Shimura \cite{specialvaluesShimura,periodsShimura} and Manin \cite{Manin1973} (independently), and the case $m=2$ was proved by Zagier \cite{Zagier1977M} and Sturm \cite{Sturmspecialvalues}.

We can view the values of $L^{\ast}_{m,f}(s)$ at the critical points as linear functionals on $f\in S_k$. Let $\mathcal{H}_k$ be the set of normalized Hecke eigenforms in $S_k$. For each $f=\sum_{g\in\mathcal{H}_k}c_g g\in S_k$ and a critical point $n+1$ for $L^{\ast}_{m,f}(s)$, we define
\begin{align} 
    \mathfrak{r}_{n,\vee^m}(f):=\sum_{g\in\mathcal{H}_k}c_gL^{\ast}_{m,g}(n+1),\label{eq:defofperiods}
\end{align}
which makes $\mathfrak{r}_{n,\vee^m}$ an element in the dual space $S_{k}^{\ast}:=\Hom(S_{k},\mathbb{C})$.
These linear functionals are expected to carry interesting information. For example, if $m=1$, then 
$$\mathfrak{r}_{n,\vee^1}(f)=\int_{0}^{\infty}f(it)t^ndt=\frac{n!}{(2\pi)^{n+1}}L(f,n+1)$$
is the $n^{{\rm th}}$ period of $f$, see \cite{Kohnen1984}. These periods are subject to many linear dependence relations, called the Eichler-Shimura relations \cite{Manin1973}; more precisely, for $0\leq t\leq k-2$, we have
\begin{align}
    \quad \mathfrak{r}_t(f)+(-1)^{t}\mathfrak{r}_{k-2-t}(f)=0,\label{eq:ES1}\tag{ES1}
    \\\quad(-1)^{t}\mathfrak{r}_t(f)+\sum_{\substack{0\leq m\leq t\\m\equiv0\mod2}}\binom{t}{m}\mathfrak{r}_{k-2-t+m}(f)+\sum_{\substack{0\leq m\leq k-2-t\\m\equiv t\mod2}}\binom{k-2-t}{m}\mathfrak{r}_m(f)=0,\label{eq:ES2}\tag{ES2}
    \\\quad\sum_{\substack{1\leq m\leq t\\m\equiv1\mod2}}\binom{t}{m}\mathfrak{r}_{k-2-t+m}(f)+\sum_{\substack{0\leq m\leq k-2-t\\m\not\equiv t\mod2}}\binom{k-2-t}{m}\mathfrak{r}_{m}(f)=0,\label{eq:ES3}\tag{ES3}
    \end{align}
where $\mathfrak{r}_{t}(f)=i^{t+1}\mathfrak{r}_{t,\vee^1}(f)$. But not much seems to be known about the linear relations among $\mathfrak{r}_{n,\vee^m}$ for $m\geq2$ since the corresponding motives are mysterious. 

Another interesting aspect of these functionals or periods is the linear independence of a subset of $\{\mathfrak{r}_{n,\vee^m}\}_n$. For example, even in the case of $m=1$, it is not simple to know whether a certain set of periods is linearly independent or not by only looking at the period relations \eqref{eq:ES1}, \eqref{eq:ES2} and \eqref{eq:ES3}. 
Recently, when $m=1$, Lei et al. \cite{oddperiods,evenperiods} have provided some evidence for the linear independence of odd periods and even periods for modular forms, respectively. Furthermore, the linear independence between one odd and one even periods has been addressed in \cite{Xueoddandevenperiods}.
In this paper, when $m=2$, we shall extend the ideas of \cite{oddperiods,evenperiods} to provide some evidence for the linear independence of a subset of the periods $\{\mathfrak{r}_{n,\vee^2}\}_n$. First, note that the critical points for the symmetric square $L$-function $L^{\ast}_{2,f}(s)$ are odd integers in the set $\{1,2,\cdots,k-1\}$ and even integers in the set $\{k,k+1,\cdots, 2k-2\}$; see Hida \cite[p.~96]{Hida1990}. Thus, via the definition \eqref{eq:defofperiods} and the functional equation \eqref{eq:functionalequationforsquareL}, we will focus on  the ``critical" periods $\mathfrak{r}_{2\ell,\vee^2}$ for integers $0\le \ell\le \frac{k-2}{2}$. Also, due to Lemma \ref{lem:integralperiods}, we will need to confine ourselves to periods $\mathfrak{r}_{2\ell,\vee^2}$ for $0\le\ell\le\frac{k-4}{2}$.

We now state our results. 
\begin{theorem}\label{thm:asymtotic}
Let $n\geq1$ be an integer. For sufficiently large $k$, if $0\leq\ell_1<\ell_2<\cdots<\ell_n\leq\frac{k-4}{2}$  are integers, then the set of periods $\{\mathfrak{r}_{2\ell_i,\vee^2}\}_{i=1}^n$ is linearly independent. 
\end{theorem}
\begin{remark}
\begin{enumerate} 
\item In Proposition \ref{prop:nonsingularityasymptotic}, we give an explicit lower bound of $k$ in Theorem \ref{thm:asymtotic}.
\item We will also prove the following complete result for $n=2$.
\end{enumerate}
\end{remark}
\begin{theorem}\label{thm:twolinearindependent}Let $0\leq\ell_1<\ell_2\leq\frac{k-4}{2}$ be integers. If $\dim S_{k}\geq 2$, then $\mathfrak{r}_{2\ell_1,\vee^2}$ and $\mathfrak{r}_{2\ell_2,\vee^2}$ are linearly independent.   
\end{theorem}
If we require that the periods $\{\mathfrak{r}_{2\ell_i,\vee^2}\}_{i=1}^n$ are ``far away" from the ``central" period $\mathfrak{r}_{\frac{k-4}{2},\vee^2}$, for example,  if $0\leq\ell_1<\ell_2<\cdots<\ell_n\leq\frac{k-2}{4}$, then we can pick up to $\frac{\log k}{4}$ linearly independent periods. More precisely, we proved the following result.
\begin{theorem}\label{thm:firstlogk}
    Let $n\geq1$ be an integer. 
    For $k\geq e^{12}$, if $0\leq\ell_1<\ell_2<\cdots<\ell_{n}\leq\frac{k-2}{4}$ are integers and $n\leq\frac{\log k}{4}$, then the set of periods $\{\mathfrak{r}_{2\ell_i,\vee^2}\}_{i=1}^{n}$ is linearly independent.
\end{theorem}

We sketch the main idea of the proofs. A natural idea is to consider the ``dual'' of the periods, which are certain modular forms. Thanks to the work of  Cohen \cite{Cohen'smodularformC_k} and Zagier \cite{Zagier1977M}, we can find the integral representation of $\mathfrak{r}_{2\ell,\vee^2}$ and reduce the linear independence among the periods $\{\mathfrak{r}_{2\ell_i,\vee^2}\}_{i=1}^n$  to the linear independence among modular forms $\{c_{k,k-1-2\ell_i}\}_{i=1}^n$ defined in \eqref{eq:defofc_k,r}. Next, following the idea of \cite{oddperiods} and \cite{evenperiods}, we consider the matrix formed by the Fourier coefficients of $\{c_{k,k-1-2\ell_i}\}_{i=1}^n$, which reduces the linear independence among $\{c_{k,k-1-2\ell_i}\}_{i=1}^n$ to the nonsingularity of the matrix.

The paper is organized as follows. In Section \ref{sect:integralrepofperiods}, we state the integral representation of the periods and give the equivalence between the linear independence among periods and the linear independence among the corresponding modular forms. In Section \ref{sect:Prelim}, we first provide some preparatory lemmas that will be used to estimate the Fourier coefficients of $c_{k,r}$.  Then we give the asymptotics of the Fourier coefficient of $c_{k,r}$. In Section \ref{sect:matrix}, we define the matrix that is formed by the Fourier coefficients of $\{c_{k,k-1-2\ell_i}\}_{i=1}^n$. Moreover, we give a criterion for the nonsingularity of this matrix. In Section \ref{sect:proofs}, we prove Theorems \ref{thm:asymtotic} and \ref{thm:firstlogk} by showing the corresponding  matrices are nonsingular. In the last section, we discuss a potential way to improve Theorems \ref{thm:asymtotic} and \ref{thm:firstlogk} and propose a related conjecture.

\section{Integral representation of the periods}\label{sect:integralrepofperiods}
Following Cohen \cite{Cohen'smodularformC_k} and Zagier \cite{Zagier1977M}, we recall the integral representation of the periods $\mathfrak{r}_{2\ell,\vee^2}$ \eqref{eq:defofperiods}. Let $k\geq4$ be an even integer and $r$ be an odd integer such that $3\leq r\leq k-1$. Define
\begin{align}
    c_{k,r}(z)=\sum_{m=0}^{\infty}\left(\sum_{t\in\mathbb{Z},t^{2}\leq4m}^{\infty}P_{k,r}(t,m)H(r,4m-t^2)\right)e^{2\pi imz},\label{eq:defofc_k,r}
\end{align}
where $P_{k,r}(t,m)$ is the polynomial defined by \cite[p.114]{Zagier1977M}
\begin{align}
    P_{k,r}(t,m)={\rm coefficient~of~}x^{k-r-1}~{\rm in~}\frac{1}{(1-tx+mx^2)^{r}}.
\end{align}
Note that $P_{k,r}(t,m)=P_{k,r}(-t,m)$.  The generalized Hurwitz class number $H(r,N)$ is defined as follows (see \cite[p.272-273]{Cohen'smodularformC_k}).  Let $r$ and $N$ be non-negative integers with $r\geq1$. For $N\geq1$, define
\begin{align}
    h(r,N)=\begin{cases}
        \frac{(-1)^{[r/2]}(r-1)!N^{r-1/2}}{2^{r-1}\pi^r}L(r,\chi_{(-1)^rN}) &(-1)^rN\equiv0,1\pmod 4,\\
        0&N\equiv 2,3\pmod4,
    \end{cases}
\end{align}
where we write $\chi_D$ for the quadratic character given by the Kronecker symbol $\chi_{D}(d)=(\frac{D}{d})$. For $N\geq0$, we define
\begin{align}
    H(r,N)=\begin{cases}
        \sum_{d^2\mid N}h(r,N/d^2)&(-1)^rN\equiv0,1\pmod4, N>0,\\
        \zeta(1-2r)&N=0,\\
        0&{\rm otherwise}.
    \end{cases}
\end{align}
If we set $(-1)^rN=Df^2$ with $D$ a fundamental discriminant (we allow $D=1$ as a fundamental discriminant), then
\begin{align}
    H(r,N)&=L(1-r,\chi_D)\sum_{d\mid f}\mu(d)\chi_D(d)d^{r-1}\sigma_{2r-1}(f/d)\\&=\frac{(-1)^{(r-1)/2}(r-1)!}{2^{r-1}\pi^r}N^{r-1/2}L(r,\chi_D)\sum_{d\mid f}\mu(d)\chi_D(d)d^{r-1}\sigma_{2r-1}(f/d),\label{eq:defofHurwitzclassnumber}
\end{align}
where $\mu$ is the M\"{o}bius function and $\sigma_{2r-1}(n)=\sum_{d>0,d\mid n}d^{2r-1}$.

\begin{proposition}$($\cite[Theorem 6.2]{Cohen'smodularformC_k}, \cite[p.114]{Zagier1977M}$)$
    Let $k\geq4$ be an even integer and $r$ be an odd integer such that $3\leq r\leq k-1$. Then the function $c_{k,r}$ in \eqref{eq:defofc_k,r} is a modular form of weight $k$ for the full modular group. If $r<k-1$, it is a cusp form.\end{proposition} 
    The following lemma gives an integral representation of the periods. It is the reason why we need to exclude the period $\mathfrak{r}_{k-2,\vee^2}$ (or equivalently $\mathfrak{r}_{k-1,\vee^2}$); see the remark before Theorem \ref{thm:asymtotic}.
\begin{lemma}$($\cite[Theorem 2]{Zagier1977M}$)$ \label{lem:integralperiods}
    Let $r,k$ be integers with $3\leq r\leq k-1$, $r$ odd, $k$ even. The Petersson inner product of the modular form $c_{k,r}$ with a  Hecke eigenform $f\in S_k$ is given by
    \begin{align}
        \langle f,c_{k,r}\rangle=-\frac{(r+k-2)!(k-2)!}{(k-r-1)!}\frac{1}{4^{r+k-2}\pi^{2r+k-1}}L_{2,f}(r+k-1).\label{eq:integralperiods}
    \end{align}
\end{lemma} 
Now, we prove the equivalence between the linear independence among the periods and the linear independence among the corresponding modular forms. 
\begin{proposition} \label{prop:equivalence}
  Let $n\ge1$ be an integer and let $0\leq\ell_1<\ell_2<\cdots<\ell_n\leq\frac{k-4}{2}$ be integers.  The set of periods $\{\mathfrak{r}_{2\ell_i,\vee^2}\}_{i=1}^n$ on $S_k$ is linearly independent if and only if the set of modular forms $\{c_{k,k-1-2\ell_i}\}_{i=1}^n$ in $S_k$ is linearly independent.
\end{proposition}
\begin{proof}
Letting $r=k-1-2\ell_i$ in \eqref{eq:integralperiods}, together with \eqref{eq:functionalequationforsquareL}, we get the Petersson inner product
\begin{align}
    \langle f,c_{k,k-1-2\ell_i}\rangle&=A_{k,\ell_i}\cdot L^{\ast}_{2,f}(2\ell_i+1)\\&=A_{k,\ell_i}\cdot \mathfrak{r}_{2\ell_i,\vee^2}(f)
,\label{eq:integrationrepresentationofperiods}\end{align}
where $f\in S_k$ is a Hecke eigenform, and $A_{k,\ell_i}$ is some nonzero constant depending only on $k$ and $\ell_i$. Let $\mathcal{H}_k$ denote the set of normalized Hecke eigenforms in $S_k$. Note that
\begin{align}
\sum_{i=1}^na_ic_{k,k-1-2\ell_i}=0~{\rm in}~S_k\quad{\rm if~and~only~if}\quad\sum_{i=1}^n\overline{a_i}\langle f_j,c_{k,k-1-2\ell_i}\rangle=0~{\rm for~all}~f_j\in\mathcal{H}_k.\label{eq:suminnerproduct}
\end{align}
First, assume that $\{c_{k,k-1-2\ell_i}\}_{i=1}^n$ is linearly independent. We claim that $\{\mathfrak{r}_{2\ell_i,\vee^2}\}_{i=1}^n$
is linearly independent. Suppose that
$$\sum_{i=1}^n b_i\mathfrak{r}_{2\ell_{i},\vee^2}=0\in S_k^{\ast}.$$ 
It follows that 
\begin{align}
    \sum_{i=1}^nb_i \mathfrak{r}_{2\ell_i,\vee^2}(f_j)=\sum_{i=1}^nb_iA_{k,\ell_i}^{-1}\langle f_j,c_{k-1-2\ell_i}\rangle=0~{\rm for~all}~f_j\in\mathcal{H}_k,
\end{align}
which implies that
\begin{align}
    \sum_{i=1}^n\overline{b_iA_{k,\ell_i}^{-1}}c_{k,k-1-2\ell_i}=0\quad{\rm by}\quad \eqref{eq:suminnerproduct}.
\end{align}
Since $\{c_{k,k-1-2\ell_i}\}_{i=1}^n$ is linearly independent, we get $\overline{b_iA_{k,\ell_i}^{-1}}=0$, and $b_i=0$ for all $1\leq i\leq n$. 

Conversely, suppose that $\{\mathfrak{r}_{2\ell_i,\vee^2}\}_{i=1}^n$
is linearly independent. We show that $\{c_{k,k-1-2\ell_i}\}_{i=1}^n$ is linearly independent by a similar argument. Assume that
$$\sum_{i=1}^n a_ic_{k,k-1-2\ell_i}=0\in S_k.$$ 
It follows that
\begin{align}&\sum_{i=1}^n\overline{a_i}\langle f_j,c_{k,k-1-2\ell_i}\rangle=0~{\rm for~all}~f_j\in\mathcal{H}_k \quad {\rm if~and~only~if} \\ &\sum_{i=1}^n\overline{a_i}A_{k,\ell_i}\cdot \mathfrak{r}_{2\ell_i, \vee^2}(f_j)=0~{\rm for~all}~f_j\in\mathcal{H}_k.\end{align}
Since $\{\mathfrak{r}_{2\ell_i,\vee^2}\}_{i=1}^n$
is linearly independent, we get $\overline{a_i}A_{k,\ell_i}=0$, and $a_i=0$ for all $1\leq i\leq n$. 
\end{proof}

\section{Asymptotics of Fourier coefficients  }\label{sect:Prelim}
First, we prove some preparatory lemmas that will be used to estimate the Fourier coefficients of $c_{k,r}$. The key is that for a perfect square $m$, the asymptotics of the $m^{{\rm th}}$ Fourier coefficient of $c_{k,r}$ is clear. We would like to mention that this observation is inspired by Serre's work \cite[Proposition 3]{SerretraceofHecke}, which describes the asymptotics of the trace of the Hecke operator $T_n$ on $S_k$. 
\begin{lemma}\label{lem:boundP_k,r}
    For integers $m\geq1$ and $t$ with $t^2<4m$, we have
    \begin{align}
        |P_{k,r}(t,m)|\leq\binom{k-2}{r-1}\cdot\frac{2^rm^{(k-1)/2}}{(4m-t^2)^{r/2}}.
    \end{align}
\end{lemma}
\begin{proof}
   Following Zagier's appendix to \cite{LangbookandZagierappendix}, we recall that
   \begin{align}
       \frac{1}{1-tx+mx^2}=\sum_{n=0}^{\infty}\frac{\rho^{n+1}-\overline{\rho}^{n+1}}{\rho-\overline{\rho}}x^n,
   \end{align}
   where $\rho+\overline{\rho}=t$ and $\rho\cdot\overline{\rho}=m$. We write $a_j:=\frac{\rho^{j+1}-\overline{\rho}^{j+1}}{\rho-\overline{\rho}}$. Then
   \begin{align}
       |a_j|\leq\frac{2|\rho|^{j+1}}{|\rho-\overline{\rho}|}\leq\frac{2m^{(j+1)/2}}{(4m-t^2)^{1/2}}.
   \end{align}
   Since $P_{k,r}(t,m)$ is the coefficient of $x^{k-r-1}$ in $(1-tx+mx^2)^{-r}$, we have
   \begin{align}
       |P_{k,r}(t,m)|&=\left|\sum_{\substack{i_1,...,i_r\geq0\\i_1+\cdots+i_r=k-r-1}}a_{i_1}\cdots a_{i_r}\right|\\&\leq \sum_{\substack{i_1,...,i_r\geq0\\i_1+\cdots+i_r=k-r-1}}\prod_{1\leq j\leq r}\frac{2m^{(i_j+1)/2}}{(4m-t^2)^{1/2}}\\&=\binom{k-2}{r-1}\cdot\frac{2^rm^{(k-1)/2}}{(4m-t^2)^{r/2}}.
   \end{align}
   Here we use the fact that a non-negative integer $n$ can be written as sum of $d$ many non-negative integers in $\binom{n+d-1}{d-1}$ ways.
\end{proof}

\begin{lemma}\label{lem:P_k,rperfectsquare}
    If $m$ is a perfect square and $1\leq r\leq k-3$, then 
    \begin{align}
        P_{k,r}(2m^{1/2},m)=\binom{k+r-2}{2r-1}\cdot m^{(k-r-1)/2}.
    \end{align}
\end{lemma}
\begin{proof}
Recall that $P_{k,r}(2m^{1/2},m)$ is the coefficient of $x^{k-r-1}$ in 
$$\frac{1}{(1-2m^{1/2}x+mx^2)^{r}}=\frac{1}{(1-m^{1/2}x)^{2r}}.$$ 
It is straightforward to see that
\begin{align}
    \frac{1}{(1-m^{1/2}x)^{2r}}&= \frac{m^{-(2r-1)/2}}{(2r-1)!}\frac{d^{2r-1}}{dx^{2r-1}}\left(\frac{1}{1-m^{1/2}x}\right)\\&=
    \frac{m^{-(2r-1)/2}}{(2r-1)!}\frac{d^{2r-1}}{dx^{2r-1}}\left(\sum_{n=0}^{\infty}(m^{1/2}x)^n\right)\\&=\frac{m^{-(2r-1)/2}}{(2r-1)!}\sum_{n=0}^{\infty}m^{(2r-1+n)/2}(2r-1+n)\cdots(n+1)x^n,
\end{align}
and letting $n=k-r-1$ gives the result.
\end{proof}
We also need the following estimates of the generalized Hurwitz class number $H(r,N)$.
\begin{lemma}\label{lem:boundsforclassnumber}
    Let $r\geq3$ be an odd integer and $N$ be a positive integer. Then
    \begin{align}
        |H(r,N)|\leq\frac{(r-1)!\zeta(2r-1)\zeta(r)}{2^{r-1}\pi^r}\cdot N^{2r-1/2}\log N.
    \end{align}
\end{lemma}
\begin{proof}
We write $-N=Df^2$ with $D$ a fundamental discriminant. First, we give an upper bound of $L(r,\chi_D)$. By Abel's partial summation formula (for example, see \cite[Theorem 6.8]{ElementarymethodsNT}), we have 
\begin{align}
    \sum_{n\leq x}\chi_D(n)n^{-r}=x^{-r}\sum_{n\leq x}\chi_D(n)+r\int_{1}^x\left(\sum_{n\leq t}\chi_D(t)\right)\frac{dt}{t^{r+1}}.
\end{align}
Note that $\chi_D$ is a primitive Dirichlet character with conductor $-D$ since $D$ is a fundamental discriminant. By P\'olya's inequality \cite[Theorem 8.21]{ApostolNT},
\begin{align}
    \left|\sum_{n\leq x}\chi_D(n)\right|\leq\sqrt{-D}\log(-D).
\end{align}
Letting $x$ go to infinity, we have
\begin{align}
 |L(r,\chi_D)|&\leq\left(r\int_{1}^{\infty}\frac{dt}{t^{r+1}}\right)\cdot\sqrt{-D}\log(-D)\\&\leq\sqrt{-D}\log(-D).
\end{align}
On the other hand, we have 
\begin{align}
    \left|\sum_{d\mid f}\mu(d)\chi_D(d)d^{r-1}\sigma_{2r-1}(f/d)\right|&\leq\sum_{d\mid f}d^{r-1}\zeta(2r-1)(f/d)^{2r-1}\\&= f^{2r-1}\zeta(2r-1)\sum_{d\mid f}d^{-r}\\&\leq f^{2r-1}\zeta(2r-1)\zeta(r),
\end{align}
where we used $\sigma_{\alpha}(n)\leq\zeta(\alpha)n^{\alpha}$ for $\alpha>1$ in the first inequality, see \cite[p.245, Exercise 9]{ElementarymethodsNT}. It follows that
\begin{align}
    |H(r,N)|&=\left|\frac{(-1)^{(r-1)/2}(r-1)!}{2^{r-1}\pi^r}N^{r-1/2}L(r,\chi_D)\sum_{d\mid f}\mu(d)\chi_D(d)d^{r-1}\sigma_{2r-1}(f/d)\right|\\&\leq\frac{(r-1)!}{2^{r-1}\pi^r}N^{r-1/2}\cdot(\sqrt{-D}\log(-D))\cdot f^{2r-1}\zeta(2r-1)\zeta(r)\\&=\frac{(r-1)!\zeta(2r-1)\zeta(r)}{2^{r-1}\pi^r}N^{r-1/2}(\sqrt{-D}\log(-D))\cdot\left(\frac{N}{-D}\right)^{r-1/2}\\&\leq \frac{(r-1)!\zeta(2r-1)\zeta(r)}{2^{r-1}\pi^r}N^{r-1/2}(\sqrt{N}\log N)\cdot N^{r-1/2}\\&=\frac{(r-1)!\zeta(2r-1)\zeta(r)}{2^{r-1}\pi^r}\cdot N^{2r-1/2}\log N,
\end{align}
as desired.
\end{proof}
\begin{lemma}\label{lem:zetavaluesatoddnegative}
For odd integer $r\geq1$, we have
\begin{align}
    |H(r,0)|=|\zeta(1-2r)|=\frac{2(2r-1)!\zeta(2r)}{(2\pi)^{2r}}.
\end{align}
\end{lemma}
    \begin{proof}
     
Recall the functional equation of $\zeta(s)$ (for example, see \cite[Theorem 12.7]{ApostolNT}), 
\begin{align}
\zeta(1-s)=2(2\pi)^{-s}\Gamma(s)\cos\left(\frac{\pi s}{2}\right)\zeta(s). \label{eq:functionalzeta}
\end{align}
Letting $s=2r$ in \eqref{eq:functionalzeta} gives the result.
    \end{proof}
 To simplify the calculation, we normalize $c_{k,r}$ so that its first Fourier coefficient is $1$. For even integer $k$  and odd integer $r$ with $3\leq r\leq k-3$, we write
\begin{align}
    f_{k,r}:=\frac{c_{k,r}}{\sum_{t^2\leq4}P_{k,r}(t,1)H(r,4-t^2)}\label{eq:defoffr}
\end{align}
and denote by $a_{k,r}(m)$ the $m^{{\rm th}}$ Fourier coefficient of $f_{k,r}$. By the definition of $c_{k,r}$ \eqref{eq:defofc_k,r} and Lemma \ref{lem:P_k,rperfectsquare}, we have
\begin{align}
    a_{k,r}(m)&=\frac
    {P_{k,r}(2m^{1/2},m)H(r,0)+\sum_{t^2<4m}P_{k,r}(t,m)H(r,4m-t^2)}{P_{k,r}(2,1)H(r,0)+\sum_{t^2<4}P_{k,r}(t,1)H(r,4-t^2)}\\&=\frac
    {\binom{k+r-2}{2r-1}\cdot m^{(k-r-1)/2}H(r,0)+\sum_{t^2<4m}P_{k,r}(t,m)H(r,4m-t^2)}{\binom{k+r-2}{2r-1}H(r,0)+\sum_{t^2<4}P_{k,r}(t,1)H(r,4-t^2)}.\label{eq:akrmiddlestep}
\end{align}
We shall show in Proposition \ref{prop:estimateofFourierco} below that $a_{k,r}(m)$ is asymptotically like $m^{(k-r-1)/2}$ when $k$ becomes large. Thus, we define
\begin{align}
    \delta_{k,m,r}:=\frac{a_{k,r}(m)}{m^{(k-r-1)/2}}-1,
\end{align}
which is the error term of $a_{k,r}(m)$. The next proposition describes the asymptotics of $a_{k,r}(m)$. 
\begin{proposition}\label{prop:estimateofFourierco}
Let $m\geq1$ be a perfect square.
If $k\geq776$ then
\begin{align}
|\delta_{k,m,r}|\leq5\left(\frac{32\pi}{k-1}\right)^rm^{2r}\log4m.
\end{align} 
\end{proposition}
\begin{proof}
By \eqref{eq:akrmiddlestep}, we have
\begin{align}
    \frac{a_{k,r}(m)}{m^{(k-r-1)/2}}
    =&\frac{\binom{k+r-2}{2r-1}H(r,0)+m^{-(k-r-1)/2}\sum_{t^2<4m}P_{k,r}(t,m)H(r,4m-t^2)}{\binom{k+r-2}{2r-1}H(r,0)+\sum_{t^2<4}P_{k,r}(t,1)H(r,4-t^2)}\\=&\frac{1+A_1}{1+A_2},
\end{align}
where $A_1$ and $A_2$ are given by
\begin{align}
    A_1:&=\frac{\sum_{t^2<4m}P_{k,r}(t,m)H(r,4m-t^2)}{\binom{k+r-2}{2r-1}H(r,0)m^{(k-r-1)/2}},\\A_2:&=\frac{\sum_{t^2<4}P_{k,r}(t,1)H(r,4-t^2)}{\binom{k+r-2}{2r-1}H(r,0)}.
\end{align}
It follows that when $|A_2|<1$, we have
\begin{align}|\delta_{k,m,r}|=\left|\frac{A_1-A_2}{1+A_2}\right|\leq\frac{|A_1|+|A_2|}{1-|A_2|}.\label{eq:md}\end{align}
We first estimate $A_1$. By Lemma \ref{lem:boundP_k,r} and \ref{lem:boundsforclassnumber}, we have
\begin{align}
|A_1|&\leq\frac{\sum_{t^2<4m}\binom{k-2}{r-1}\frac{2^rm^{(k-1)/2}}{(4m-t^2)^{r/2}}\cdot\frac{(r-1)!\zeta(2r-1)\zeta(r)}
   {2^{r-1}\pi^r}\cdot(4m-t^2)^{2r-1/2}\log(4m-t^2)}{\binom{k+r-2}{2r-1}|H(r,0)|m^{(k-r-1)/2}}   
\\&=\frac{(r-1)!\binom{k-2}{r-1}}{\binom{k+r-2}{2r-1}|H(r,0)|}\cdot\frac{2\zeta(2r-1)\zeta(r)m^{r/2}}{\pi^r}\sum_{t^2<4m}(4m-t^2)^{(3r-1)/2}\log(4m-t^2)\\&\leq\frac{(r-1)!\binom{k-2}{r-1}}{\binom{k+r-2}{2r-1}|H(r,0)|}\cdot\frac{2\zeta(5)\zeta(3)m^{r/2}}{\pi^r}\cdot2(4m)^{1/2}(4m)^{(3r-1)/2}\log4m\\&=\frac{(r-1)!\binom{k-2}{r-1}}{\binom{k+r-2}{2r-1}|H(r,0)|}\cdot\frac{4\zeta(5)\zeta(3)2^{3r}m^{2r}\log4m}{\pi^r}.
\end{align}
By Lemma \ref{lem:zetavaluesatoddnegative}, we have 
\begin{align}
    \frac{(r-1)!\binom{k-2}{r-1}}{\binom{k+r-2}{2r-1}|H(r,0)|}&=\frac{(k-2)!(2r-1)!}{(k+r-2)!}\cdot\frac{(2\pi)^{2r}}{
    2(2r-1)!\zeta(2r)}\\&=\frac{(k-2)(k-3)\cdots1}{(k+r-2)\cdots(k-1)(k-2)\cdots1}\cdot\frac{(2\pi)^{2r}}{2\zeta(2r)}\\&\leq\frac{1}{(k-1)^{r}}\cdot\frac{(2\pi)^{2r}}{2}.
\end{align}
It follows that 
\begin{align}
    |A_1|&\leq\frac{(2\pi)^{2r}}{2(k-1)^r}\cdot\frac{4\zeta(5)\zeta(3)2^{3r}m^{2r}\log4m}{\pi^r}\\&=2\zeta(5)\zeta(3)\left(\frac{32\pi}{k-1}\right)^rm^{2r}\log4m,\label{eq:bdA1}
\end{align}
Note that the above argument in particular implies that 
\begin{align}
    |A_{2}|\leq2\zeta(5)\zeta(3)\left(\frac{32\pi}{k-1}\right)^r\log4\leq2\zeta(5)\zeta(3)\left(\frac{32\pi}{k-1}\right)^rm^{2r}\log4m.\label{eq:roughboundforA2}
\end{align}
On the other hand, the upper bound of $|A_2|$ can be obtained by
\begin{align}
    |A_2|=&\left|\frac{P_{k,r}(0,1)+2P_{k,r}(1,1)H(r,3)}{\binom{k+r-2}{2r-1}H(r,0)}\right|\\\leq&\left|\frac{P_{k,r}(0,1)}{\binom{k+r-2}{2r-1}H(r,0)}\right|+\left|\frac{2P_{k,r}(1,1)H(r,3)}{\binom{k+r-2}{2r-1}H(r,0)}\right|\\\leq&\frac{\binom{k-2}{r-1}\frac{2^r}{4^{r/2}}\frac{(r-1)!\zeta(2r-1)\zeta(r)}{2^{r-1}\pi^r}4^{2r-1/2}\log4}{\binom{k+r-2}{2r-1}|H(r,0)|}+2\frac{\binom{k-2}{r-1}\frac{2^r}{3^{r/2}}\frac{(r-1)!\zeta(2r-1)\zeta(r)}{2^{r-1}\pi^r}3^{2r-1/2}\log3}{\binom{k+r-2}{2r-1}|H(r,0)|}\\\leq&\frac{(r-1)!\binom{k-2}{r-1}}{\binom{k+r-2}{2r-1}|H(r,0)|}\left(\frac{2^r\zeta(2r-1)\zeta(r)4^{2r-1/2}\log4}{4^{r/2}2^{r-1}\pi^r}+2\frac{2^r\zeta(2r-1)\zeta(r)3^{2r-1/2}\log3}{3^{r/2}2^{r-1}\pi^r}\right)\\\leq&\frac{1}{(k-1)^{r}}\cdot\frac{(2\pi)^{2r}}{2}\left(\frac{\zeta(2r-1)\zeta(r)(\log4)4^{3r/2}}{\pi^r}+\frac{4}{\sqrt{3}}\frac{\zeta(2r-1)\zeta(r)(\log3)3^{3r/2}}{\pi^r}\right)\\\leq&\frac{\zeta(5)\zeta(3)\log4}{2}\left(\frac{32\pi}{k-1}\right)^r+\frac{2\zeta(5)\zeta(3)\log3}{\sqrt{3}}\left(\frac{12\sqrt{3}\pi}{k-1}\right)^r\label{eq:bdA2}
\end{align}
Since $k\geq776$, we have
\begin{align}
|A_2|&\leq\frac{\zeta(5)\zeta(3)\log4}{2}\left(\frac{32\pi}{776-1}\right)^3+\frac{2\zeta(5)\zeta(3)\log3}{\sqrt{3}}\left(\frac{12\sqrt{3}\pi}{776-1}\right)^3<1.
\end{align}
Plugging the bounds \eqref{eq:bdA1}, \eqref{eq:roughboundforA2} and \eqref{eq:bdA2} into \eqref{eq:md} gives
\begin{align}
|\delta_{k,m,r}|&\leq2\cdot\frac{2\zeta(5)\zeta(3)\left(\frac{32\pi}{k-1}\right)^rm^{2r}\log4m}{1-\frac{\zeta(5)\zeta(3)\log4}{2}\left(\frac{32\pi}{k-1}\right)^r-\frac{2\zeta(5)\zeta(3)\log3}{\sqrt{3}}\left(\frac{12\sqrt{3}\pi}{k-1}\right)^r}\\&\leq \frac{4\zeta(5)\zeta(3)}{1-\frac{\zeta(5)\zeta(3)\log4}{2}\left(\frac{32\pi}{776-1}\right)^3-\frac{2\zeta(5)\zeta(3)\log3}{\sqrt{3}}\left(\frac{12\sqrt{3}\pi}{776-1}\right)^3}\cdot\left(\frac{32\pi}{k-1}\right)^rm^{2r}\log4m\\&\leq4.999942\left(\frac{32\pi}{k-1}\right)^rm^{2r}\log4m,
\end{align}
as desired.
\end{proof}
\section{Criterion for the linear independence}\label{sect:matrix}
Following the idea of \cite{oddperiods} and \cite{evenperiods}, we define matrices that are used to study the linear independence of $\{c_{k,k-1-2\ell_i}\}_{i=1}^n$, see \eqref{eq:defofc_k,r} and Proposition \ref{prop:equivalence}. Let $k\geq4$ be an even integer and $0\leq\ell_1<\ell_2<\cdots<\ell_n\leq\frac{k-4}{2}$ be integers. We pick out the Fourier coefficients $a_{k,k-1-2\ell_i}(1), a_{k,k-1-2\ell_i}(4), a_{k,k-1-2\ell_i}(4^2),...,a_{k,k-1-2\ell_i}(4^{n-1})$ of each $f_{k,k-1-2\ell_i}(1\leq i\leq n)$  \eqref{eq:defoffr} and place them in a $(n,n)$-matrix:
\begin{align}
    A_{k;\ell_1,...,\ell_n}:&=\left[a_{k,k-1-2\ell_i}(4^{j-1})\right]_{1\leq i,j\leq n}\\&=\left[4^{\ell_i(j-1)}(1+\delta_{k,4^{j-1},k-1-2\ell_i})\right]_{1\leq i,j\leq n} .\label{eq:defofmatrix}
\end{align}
To show that $\{c_{k,k-1-2\ell_i}\}_{i=1}^n$ is linearly independent, it suffices to show that $A_{k;\ell_1,...,\ell_n}$ is non-singular.
We write $$\varepsilon_{k,i,j}:=\delta_{k,4^{j-1},k-1-2\ell_i}.$$
Such a choice of $A_{k;\ell_1,...,\ell_n}$ allows us to relate it to the well-known Vondermonde matrix. More precisely, we ignore the error terms $\varepsilon_{k,i,j}$ of $A_{k;\ell_1,...,\ell_n}$  and define
\begin{align}
\mathscr{A}_{k;\ell_1,...,\ell_n}:=\begin{bmatrix}
        1 & 4^{\ell_1} & (4^{\ell_1})^2 & \cdots & (4^{\ell_1})^{n-1}\\
        1 & 4^{\ell_2} & (4^{\ell_2})^2 & \cdots & (4^{\ell_2})^{n-1}\\
        \vdots & \vdots & \vdots & \ddots & \vdots\\
        1 & 4^{\ell_{n-1}} & (4^{\ell_{n-1}})^2 & \cdots & (4^{\ell_{n-1}})^{n-1}\\ 
        1 & 4^{\ell_n} & (4^{\ell_n})^2 & \cdots & (4^{\ell_n})^{n-1}
    \end{bmatrix}.
\end{align}
Thus, similar to \cite[(2.8)]{oddperiods}, we have
\begin{align}
    \det\mathscr{A}_{k;\ell_1,...,\ell_n}&=\prod_{1\leq i<j\leq n}(4^{\ell_j}-4^{\ell_i})\\&=\prod_{1\leq i<j\leq n}4^{\ell_j}(1-4^{\ell_i-\ell_j})\\&\geq\prod_{1\leq i<j\leq n}\frac{3}{4} \cdot 4^{\ell_j}\\&=\left(\frac{3}{4}\right)^{n(n-1)/2}\prod_{j=1}^n4^{\ell_j(j-1)}.\label{eq:lowerboundofvandermondedet}
\end{align}
\begin{remark}
    In fact, if $m$ is a perfect square, then the asymptotics of $a_{k,k-1-2\ell_i}$ is clear (see Proposition \ref{prop:estimateofFourierco}), and the $(n,n)$-matrix  
    \begin{align}A_{k;\ell_1,...,\ell_n}:&=\left[a_{k,k-1-2\ell_i}(m^{j-1})\right]_{1\leq i,j\leq n}\\&=\left[m^{\ell_i(j-1)}(1+\delta_{k,m^{j-1},k-1-2\ell_i})\right]_{1\leq i,j\leq n}\end{align}
    is related to the Vandermonde matrix
\begin{align}
\mathscr{A}_{k;\ell_1,...,\ell_n}:=\begin{bmatrix}
        1 & m^{\ell_1} & (m^{\ell_1})^2 & \cdots & (m^{\ell_1})^{n-1}\\
        1 & m^{\ell_2} & (m^{\ell_2})^2 & \cdots & (m^{\ell_2})^{n-1}\\
        \vdots & \vdots & \vdots & \ddots & \vdots\\
        1 & m^{\ell_{n-1}} & (m^{\ell_{n-1}})^2 & \cdots & (m^{\ell_{n-1}})^{n-1}\\ 
        1 & m^{\ell_n} & (m^{\ell_n})^2 & \cdots & (m^{\ell_n})^{n-1}
    \end{bmatrix},\end{align}
which has determinant
\begin{align}
    \det\mathscr{A}_{k;\ell_1,...,\ell_n}&=\prod_{1\leq i<j\leq n}(m^{\ell_j}-m^{\ell_i})\\&=\prod_{1\leq i<j\leq n}m^{\ell_j}(1-m^{\ell_i-\ell_j})\\&\geq\prod_{1\leq i<j\leq n}\frac{m-1}{m} \cdot m^{\ell_j}\\&=\left(\frac{m-1}{m}\right)^{n(n-1)/2}\prod_{j=1}^nm^{\ell_j(j-1)}.
\end{align}
Any such choice of matrix $A_{k;\ell_1,...,\ell_n}$ will suffice to study the linear independence of $\{c_{k,k-1-2\ell_i}\}_{i=1}^{n}$. However, the lower bounds of $k$ in Propositions \ref{prop:nonsingularityasymptotic} and \ref{prop:nonsingularfarfromcenter} will not be significantly influenced by the choice of $m$, see Proposition \ref{prop:criterionofnonsingularity}. So we choose $m=4$ to ease the notation. 
\end{remark}The rest of this section is devoted to giving a criterion for the nonsingularity of $A_{k;\ell_1,...,\ell_n}$.
We have the following observation, see also \cite[Lemma 2.4]{oddperiods} and \cite[Corollary 2.5]{oddperiods}.
\begin{lemma}\label{lem:maxofsimplefactorofdet}Let $S_n$ denote the permutation group on the set $\{1,...,n\}$ and $\tau\in S_n$. Then 
\begin{align}
    \prod_{1\leq i\leq n}4^{\ell_i(\tau(i)-1)}
\end{align}
attains maximum when $\tau$ is the identity permutation in $S_n$.    
\end{lemma}
\begin{proof}
  It suffices to show that for integers $\ell_2>\ell_1\geq0$ and $n_2>n_1\geq1$ we have 
  \begin{align}
      4^{\ell_2(n_2-1)}\cdot 4^{\ell_1(n_1-1)}>4^{\ell_2(n_1-1)}\cdot 4^{\ell_1(n_2-1)}
  ,\end{align}
  which is clear from
  \begin{align}
      \ell_2(n_2-1)+\ell_1(n_1-1)-\ell_2(n_1-1)-\ell_1(n_2-1)=(\ell_2-\ell_1)(n_2-n_1)>0.
  \end{align}
  This finishes the proof.
\end{proof}
Combining the inequality \eqref{eq:lowerboundofvandermondedet} and Lemma \ref{lem:maxofsimplefactorofdet}, we have the following.
\begin{corollary}
    For any $\tau\in S_{n}$, we have 
    \begin{align}
        \prod_{i=1}^{n}4^{\ell_i(\tau(i)-1)}\leq\left(\frac{4}{3}\right)^{n(n-1)/2}\cdot\det\mathscr{A}_n
    .\label{eq:boundofsimplefactor}\end{align}
\end{corollary}
Now, we give the criterion for the nonsingularity of $A_{k;\ell_1...,\ell_n}$.
\begin{proposition}\label{prop:criterionofnonsingularity}\cite[Lemma 2.7]{oddperiods}
    Let $A_{k;\ell_1,...,\ell_n}$ be as in \eqref{eq:defofmatrix}. Define a function 
    \begin{align}f(x_1,...,x_n):=\left(\prod_{1\leq i\leq n}(1+x_{i})\right)-1.\label{eq:defoff}\end{align}
    Suppose that
    \begin{align}
        \sup_{\tau\in S_n}|f(\varepsilon_{k,1,\tau(1)},...,\varepsilon_{k,n,\tau(n)})|<\left(\frac{3}{4}\right)^{n(n-1)/2}\cdot(n!)^{-1}.\label{eq:conditionfordetneq0}
    \end{align}
    Then $A_{k;\ell_1,...,\ell_n}$ is nonsingular.
\end{proposition}
\begin{proof}
Note that 
 \begin{align}
     \det A_{k;\ell_1,...,\ell_n}&=\sum_{\tau\in S_n}(-1)^{\sign(\tau)}\prod_{i=1}^{n}4^{\ell_i(\tau(i)-1)}(1+\varepsilon_{k;i,\tau(i)})\\&=\det\mathscr{A}_{k;\ell_1,...,\ell_n}+\sum_{\tau\in S_n}(-1)^{\sign(\tau)}f(\varepsilon_{k,1,\tau(1)},...,\varepsilon_{k,n,\tau(n)})\prod_{i=1}^n4^{\ell_i(\tau(i)-1)}\\&\geq\det\mathscr{A}_{k;\ell_1,...,\ell_n}\left(1-n!\cdot\left(\frac{4}{3}\right)^{n(n-1)/2}\cdot  \sup_{\tau\in S_n}|f(\varepsilon_{k,1,\tau(1)},...,\varepsilon_{k,n,\tau(n)})|\right),
 \end{align}
 where the inequality comes from \eqref{eq:boundofsimplefactor}. Thus, $\eqref{eq:conditionfordetneq0}$ implies that $\det A_{k;\ell_1,...,\ell_n}>0$.
 \end{proof}

\section{Proof of results}\label{sect:proofs}
In this section, we prove Theorems \ref{thm:asymtotic}, \ref{thm:twolinearindependent} and \ref{thm:firstlogk}. It suffices to show the nonsingularity of $A_{k;\ell_1,...,\ell_n}$ in each case. The key is to manipulate $\sup_{\tau\in S_n}|f(\varepsilon_{k,1,\tau(1)},...,\varepsilon_{k,n,\tau(n)})|$ so that the condition \eqref{eq:conditionfordetneq0} of Proposition \ref{prop:criterionofnonsingularity} is satisfied. First, we prove an elementary inequality.
\begin{lemma}\label{lem:inequalityforf}
    Let $f$ be defined as in \eqref{eq:defoff}. Suppose there is an $M>0$  such that $|x_i|\leq M$ for all $1\leq i\leq n$. Then 
    \begin{align}
        |f(x_1,...,x_n)|\leq(1+M)^{n}-1.
    \end{align}
\end{lemma}
\begin{proof}
    Note that 
    \begin{align}
|f(x_1,...,x_n)|&=\left|\prod_{1\leq i\leq n}(1+x_i)-1\right|\\&=
\left|\sum_{i=1}^nx_{i}+\sum_{1\leq i<j\leq n}x_ix_j+\cdots+x_1x_2\cdots x_n\right|\\&\leq\binom{n}{1}M+\binom{n}{2}M^2+\cdots+\binom{n}{n}M^n\\&=(1+M)^n-1,\label{eq:boundoff}\end{align}
as desired.
\end{proof}
Theorem \ref{thm:asymtotic} follows immediately from the result below.
\begin{proposition}\label{prop:nonsingularityasymptotic}
Let $n\geq1$ be an integer and $k$ be an even integer with \begin{align}k>\max\left(776,\frac{32\pi\sqrt[3]{20}\cdot4^{13(n-1)/6}}{\sqrt[3]{\sqrt[n]{\left(\frac{3}{4}\right)^{(n-1)n/2}(n!)^{-1}+1}-1}}+1\right)\label{eq:lowerboundforgeneralcase}\end{align}
If $0\leq\ell_1<\ell_2<\cdots<\ell_n\leq\frac{k-4}{2}$, then $A_{k;\ell_1,...,\ell_n}$ is nonsingular.
\end{proposition}
\begin{proof}
   Let $m=4^{j-1}$ and $r=k-1-2\ell_i,$ where $ 1\leq i,j\leq n$. The assumption \eqref{eq:lowerboundforgeneralcase} on $k$ allows us to use Proposition \ref{prop:estimateofFourierco}, which implies that
   \begin{align}
|\delta_{k,m,r}|&\leq5\left(\frac{32\pi}{k-1}\right)^rm^{2r}\log4m\\&\leq
 5\left(\frac{32\pi}{k-1}\right)^rm^{2r}2(4m)^{1/2}\\&\leq20\left(\frac{32\pi}{k-1}\right)^rm^{2r+r/6}\\&=20\left(\frac{32\pi\cdot4^{13(j-1)/6}}{k-1}\right)^r\\&\leq20\left(\frac{32\pi\cdot4^{13(n-1)/6}}{k-1}\right)^r.\end{align}
 Note that \eqref{eq:lowerboundforgeneralcase} in particular implies that
 \begin{align}
  \left(\frac{32\pi\cdot4^{13(n-1)/6}}{k-1}\right)^r&\leq\left(\frac{32\pi\cdot4^{13(n-1)/6}\sqrt[3]{\sqrt{2}-1}}{32\pi\sqrt[3]{20}\cdot4^{13(n-1)/6}}\right)^r<1.
 \end{align}
Then by Lemma \ref{lem:inequalityforf}
   \begin{align}&\sup_{\tau\in S_n}|f(\varepsilon_{k,1,\tau(1)},...,\varepsilon_{k,n,\tau(n)})|\\&<\left(1+\left(\frac{32\pi\sqrt[3]{20}\cdot4^{13(n-1)/6}\cdot\sqrt[3]{\sqrt[n]{\left(\frac{3}{4}\right)^{(n-1)n/2}(n!)^{-1}+1}-1}}{32\pi\sqrt[3]{20}\cdot4^{13(n-1)/6}}\right)^3\right)^n-1\\&=\left(\frac{3}{4}\right)^{n(n-1)/2}\cdot(n!)^{-1}.\end{align}
   Hence $A_{k;\ell_1,...,\ell_n}$ is nonsingular by Proposition \ref{prop:criterionofnonsingularity}.
\end{proof}

We next present the proof of Theorem \ref{thm:twolinearindependent}.
\begin{proof}[Proof of Theorem \ref{thm:twolinearindependent}]Furthermore, note that Theorem \ref{thm:asymtotic} or Proposition \ref{prop:nonsingularityasymptotic} implies that for even integer $k$ with
\begin{align}
    k&>\frac{32\pi\sqrt[3]{20}\cdot4^{13(2-1)/6}}{\sqrt[3]{\sqrt[2]{\left(\frac{3}{4}\right)^{(2-1)2/2}(2!)^{-1}+1}-1}}+1\\&\approx9880.98\\&\geq9880
\end{align}
if $\ell_1$ and $\ell_2$ are integers such that $0\leq\ell_1<\ell_2\leq\frac{k-4}{2}$, then the periods $\mathfrak{r}_{2\ell_1,\vee^2}$ and $\mathfrak{r}_{2\ell_2,\vee^2}$ are linearly independent. Denote by $c_{k,r}(m)$ the $m^{\rm th}$ Fourier coefficient of $c_{k,r}$. We verified the nonsingularity of the matrix 
\begin{align}C_{k;\ell_1,\ell_2}:=
\begin{bmatrix}
    c_{k,k-1-2\ell_1}(1) & c_{k,k-1-2\ell_1}(4)\\c_{k,k-1-2\ell_2}(1) & c_{k,k-1-2\ell_2}(4)
\end{bmatrix}    \label{eq:matrixcononormalization}
\end{align}
for $12\leq k\leq 9882, k\neq 14$, and $0\leq\ell_1<\ell_2\leq\frac{k-4}{2}$ by \href{https://www.sagemath.org}{SageMath}, see \cite{githubNi} for the codes. Thus, the proof of Theorem \ref{thm:twolinearindependent} is complete. 
\end{proof}

We last prove Theorem \ref{thm:firstlogk}. 
\begin{lemma}\label{prop:boundofsup}
Let $n\ge1$ and $k\geq e^{12}$. If $0\leq\ell_1<\cdots<\ell_n\leq\frac{k-2}{4}$ are integers and $n\leq\frac{\log k}{4}$, then
 \begin{align}
        \sup_{\tau\in S_n}|f(\varepsilon_{k,1,\tau(1)},...,\varepsilon_{k,n,\tau(n)})|\leq k^{\left[\frac{\log(k-1)-\log(32\pi\cdot4^{-13/6})}{\log k}-\frac{13\log4}{24}-\frac{2(\log(5(\log k)^2)-\log4)}{k\log k}\right]\cdot\frac{-k}{2}}.
    \end{align}
\end{lemma}
\begin{proof}
First, we find a uniform bound for $\varepsilon_{k,i,j} (1\leq i,j\leq n)$.
Let $m=4^{j-1}$ and $r=k-1-2\ell_i$. By Proposition \ref{prop:estimateofFourierco}, we have
\begin{align}
   |\varepsilon_{k,i,j}|&\leq5\left(\frac{32\pi}{k-1}\right)^rm^{2r}\log4m\\&\leq5\left(\frac{32\pi}{k-1}\right)^rm^{2r}2(4m)^{1/2}\\&\leq20\left(\frac{32\pi}{k-1}\right)^rm^{2r}\cdot m^{r/6}\\&=20\left(\frac{32\pi\cdot4^{13(j-1)/6}}{k-1}\right)^{k-1-2\ell_i}.\label{eq:errortermforl<k-2/4}
\end{align}
Since $j\leq\frac{\log k}{4}$, we have
   \begin{align}
       \frac{32\pi\cdot4^{13(j-1)/6}}{k-1}&=\frac{32\pi\cdot4^{-13/6}\cdot4^{13j/6}}{k-1}\\&\leq\frac{32\pi\cdot4^{-13/6}\cdot4^{13\log k/24}}{k-1}\\&=
       k^{-\frac{\log(k-1)-\log(32\pi\cdot4^{-13/6})}{\log k}+\frac{13\log4}{24}}.\label{eq:errortermforl<k-2/4,2}
   \end{align}
Note that in particular,
\begin{align}
    k^{-\frac{\log(k-1)-\log(32\pi\cdot4^{-13/6})}{\log k}+\frac{13\log4}{24}}\leq k^{-\frac{\log(e^{12}-1)-\log(32\pi\cdot4^{-13/6})}{\log e^{12}}+\frac{13\log 4}{24}}\leq k^{-0.115}<1.\label{eq:middlestepforuseoflem}
\end{align}
As $k-1-2\ell_i\geq \frac{k}{2}$, \eqref{eq:errortermforl<k-2/4} and \eqref{eq:errortermforl<k-2/4,2} imply that 
\begin{align}
|\varepsilon_{k,i,j}|&\leq 20k^{\left[\frac{\log(k-1)-\log(32\pi\cdot4^{-13/6})}{\log k}-\frac{13\log4}{24}\right]\cdot\frac{-k}{2}}\label{eq:boundofnewerror}
\end{align}
We now give the bound of  $\sup_{\tau\in S_n}|f(\varepsilon_{k,1,\tau(1)},...,\varepsilon_{k,n,\tau(n)})|$. To simplify the proof, we will use a different bound of $f$ defined in \eqref{eq:defoff}. By a similar argument as \eqref{eq:boundoff}, if $|x_i|\leq M$ for all $1\leq i\leq n$ and $nM\leq1$ then
\begin{align}
    |f(x_1,...,x_n)|&\leq\binom{n}{1}M+\binom{n}{2}M^2+\cdots+\binom{n}{n}M^n\leq \sum_{i=1}^n (nM)^i\leq n^2M.\label{eq:newboundoff}
\end{align}
Note that \eqref{eq:boundofnewerror} and \eqref{eq:middlestepforuseoflem} imply that 
\begin{align}
   |\varepsilon_{k,i,j}|\leq20\left(\frac{1}{k^{0.115}}\right)^{k/2}\leq \frac{4}{\log k}\leq \frac{1}{n}.
\end{align}
Hence by \eqref{eq:newboundoff} and \eqref{eq:boundofnewerror}, 
 \begin{align}
        \sup_{\tau\in S_n}|f(\varepsilon_{k,1,\tau(1)},...,\varepsilon_{k,n,\tau(n)})|&\leq n^2\cdot20k^{\left[\frac{\log(k-1)-\log(32\pi\cdot4^{-13/6})}{\log k}-\frac{13\log4}{24}\right]\cdot\frac{-k}{2}}\\&\leq\frac{5(\log k)^2}{4}\cdot k^{\left[\frac{\log(k-1)-\log(32\pi\cdot4^{-13/6})}{\log k}-\frac{13\log4}{24}\right]\cdot\frac{-k}{2}}\\&=k^{\left[\frac{\log(k-1)-\log(32\pi\cdot4^{-13/6})}{\log k}-\frac{13\log4}{24}-\frac{2(\log(5(\log k)^2)-\log4)}{k\log k}\right]\cdot\frac{-k}{2}},
    \end{align}
    as desired.
\end{proof}  
Theorem \ref{thm:firstlogk} follows from the next result.
\begin{proposition}\label{prop:nonsingularfarfromcenter}
Let $n\geq1$ and $k\geq e^{12}$. If $0\leq\ell_1<\ell_2<\cdots<\ell_n\leq\frac{k-2}{4}$ and $n\leq\frac{\log k}{4}$, then $A_{k;\ell_1,...,\ell_n}$ is nonsingular.
  \end{proposition}
  \begin{proof}
      First, we give a lower bound of the right hand side of \eqref{eq:conditionfordetneq0}. Recall the non-asymptotic Stirling's approximation \cite{Stirling's-formula}, which states that for all $n\geq1$
\begin{align} 
    \sqrt{2\pi n}\left(\frac{n}{e}\right)^ne^{\frac{1}{12n+1}}<n!<   \sqrt{2\pi n}\left(\frac{n}{e}\right)^ne^{\frac{1}{12n}}.
\end{align}
In particular, we have
\begin{align}\sqrt{2\pi}<\frac{n!e^{n}}{n^{n+\frac{1}{2}}}\leq e.\label{eq:Stirling}\end{align}
Using \eqref{eq:Stirling} and the fact that $(e/x)^x$ decreases for $x>0$, we get
      \begin{align}
          \left(\frac{3}{4}\right)^{n(n-1)/2}\cdot(n!)^{-1}&\geq\left(\frac{3}{4}\right)^{n(n-1)/2}\frac{1}{e\sqrt{2\pi n}}\left(\frac{e}{n}\right)^n\\&\geq\left(\frac{3}{4}\right)^{\frac{\log k}{4}\left(\frac{\log k}{4}-1\right)/2}\frac{2}{e\sqrt{2\pi\log k}}\left(\frac{4e}{\log k}\right)^{\frac{\log k}{4}}\\&=k^{\frac{-\log(4/3)}{4}\cdot\left(\frac{\log k}{4}-1\right)/2}\frac{2}{e\sqrt{2\pi\log k}}\left(\frac{4e}{\log k}\right)^{\frac{\log k}{4}}.
      \end{align}
It is not difficult to see that for all $k\geq e^{12}$, we have
\begin{align}
    \frac{2}{e\sqrt{2\pi\log k}}\geq\frac{1}{k^{\log(4/3)}}\quad{\rm and}\quad \left(\frac{4e}{\log k}\right)^{\frac{\log k}{4}}\geq k^{\frac{-\log(4/3)}{4}\cdot\frac{\log k}{4}},
\end{align}
which implies that
\begin{align}
          \left(\frac{3}{4}\right)^{n(n-1)/2}\cdot(n!)^{-1}&\geq k^{\frac{-\log(4/3)}{4}\cdot\left[\left(\frac{\log k}{4}-1\right)/2+4+\frac{\log k}{4}\right]}\\&=k^{\frac{\log(4/3)}{4}\cdot(-\frac{3\log k}{8}-\frac{7}{2})}\\&\geq k^{\frac{\log(4/3)}{4}\cdot\frac{-k}{2}}.
          \end{align}
On the other hand, 
 \begin{align}
     \sup_{\tau\in S_n}|f(\varepsilon_{k,1,\tau(1)},...,\varepsilon_{k,n,\tau(n)})|\leq k^{\left[\frac{\log(k-1)-\log(32\pi\cdot4^{-13/6})}{\log k}-\frac{13\log4}{24}-\frac{2(\log(5(\log k)^2)-\log4)}{k\log k}\right]\cdot\frac{-k}{2}}.    
    \end{align}
    Note that for $k\geq e^{12}$, we have
    \begin{align}
        \frac{\log(k-1)-\log(32\pi\cdot4^{-13/6})}{\log k}-\frac{13\log4}{24}-\frac{2(\log(5(\log k)^2)-\log4)}{k\log k}\geq 0.115>\frac{\log(4/3)}{4},
    \end{align}
    which implies that 
     \begin{align} \sup_{\tau\in S_n}|f(\varepsilon_{k,1,\tau(1)},...,\varepsilon_{k,n,\tau(n)})|\leq k^{0.115\cdot\frac{-k}{2}}<k^{\frac{\log(4/3)}{4}\cdot\frac{-k}{2}}\leq    \left(\frac{3}{4}\right)^{n(n-1)/2}\cdot(n!)^{-1}\end{align}
Therefore, $A_{k;\ell_1,...,\ell_n}$ is nonsingular by Proposition \ref{prop:criterionofnonsingularity}.
  \end{proof}
  \section{Discussions}\label{sect:discussion}
Since the lower bound for $k$ in Proposition \ref{prop:nonsingularityasymptotic} is gigantic, a natural question is how much we can improve it. 
As mentioned in \cite[Remark 2.9]{oddperiods}, there is another natural choice of matrix to investigate the linear independence of the periods. Recall that $a_{k,r}(m)$ denotes the $m^{{\rm th}}$ Fourier coefficient of $f_{k,r}$ defined by \eqref{eq:defoffr}. We define
 \begin{align}
    B_{k;\ell_1,...,\ell_n}:&=\left[a_{k,k-1-2\ell_i}(j^2)\right]_{1\leq i,j\leq n}\\&=
\left[j^{2\ell_i}(1+\eta_{k,i,j})\right]_{1\leq i,j\leq n},  
\end{align}
where $\eta_{k,i,j}$ goes to zero as $k$ tends to infinity. 
Then the ``limit" matrix becomes
\begin{align}
    \mathscr{B}_{k;\ell_1,...,\ell_n}:=\begin{bmatrix}
        1 & 2^{2\ell_1} & 3^{2\ell_1} & \cdots & n^{2\ell_1}\\
        1 & 2^{2\ell_2} & 3^{2\ell_2} & \cdots & n^{2\ell_2}\\
        \vdots & \vdots & \vdots & \ddots & \vdots\\
        1 & 2^{2\ell_{n-1}} & 3^{2\ell_{n-1}} & \cdots & n^{2\ell_{n-1}}\\ 
        1 & 2^{2\ell_n} & 3^{2\ell_n} & \cdots & n^{2\ell_n}
    \end{bmatrix},\label{eq:generalizaedVandermonde}
\end{align}
which is a generalized Vandermonde matrix \cite[Vol.2, p.99]{Gantmacher1959}. If one can get a good result about the lower bound of the determinant of \eqref{eq:generalizaedVandermonde}, then the upper bound for $\sup_{\tau\in S_n}|f(\varepsilon_{k,1,\tau(1)},...,\varepsilon_{k,n,\tau(n)})|$ in \eqref{eq:conditionfordetneq0} becomes bigger, which will lead to a smaller lower bound for $k$ in Proposition \ref{prop:nonsingularityasymptotic}.  It is reasonable to propose the following conjecture.
\begin{conjecture} \label{conj,1}
    Let $0\leq\ell_1<\ell_2<\cdots<\ell_n\leq\frac{k-4}{2}$. If $n\leq\dim S_k$, then the set of periods $\{\mathfrak{r}_{2\ell_i,\vee^2}\}_{i=1}^n$ is linearly independent.
\end{conjecture}

Note that Theorem \ref{thm:twolinearindependent} provides evidence for  Conjecture \ref{conj,1} when $n=2$. The case $n=1$ is also true since we verified the nonsingularity of the matrix $C_{k,\ell_1,\ell_2}$ defined in \eqref{eq:matrixcononormalization} for $12\leq k\leq 9882, k\neq14$.
There is another question related to the set of periods $\{\mathfrak{r}_{2\ell,\vee^2}\}_{\ell=0}^{(k-4)/2}$, that is, can this set span the whole space of $S_k^{\ast}$? We believe that the answer is affirmative, which also follows from Conjecture \ref{conj,1}, but we are only able to show the following corollary of Theorem \ref{thm:firstlogk}.

\begin{corollary}
    Suppose $k\ge e^{12}$ is even. Then
    \begin{align}
        \dim \Span\left( \{\mathfrak{r}_{2\ell,\vee^2}\}_{\ell=0}^{(k-4)/2}\right)\ge \left\lfloor\frac{\log k}{4}\right\rfloor.
    \end{align}
\end{corollary}
\begin{proof}
    This is a direct corollary of Theorem \ref{thm:firstlogk}.
\end{proof}

Another natural question about the periods $\mathfrak{r}_{2\ell,\vee^2}$ is whether we can find Eichler-Shimura 
type linear relations \eqref{eq:ES2} and \eqref{eq:ES3} among them. We hope to take it up in a future work.

\section*{Ackowledgements}
The authors would like to thank the referee for detailed comments and insightful advice.
\bibliographystyle{plain}
  
  \vspace{.2in}


\providecommand{\bysame}{\leavevmode\hbox
to3em{\hrulefill}\thinspace}

\bibliography{ref}
\end{document}